\documentclass{ruthesis}
\usepackage{fancyhdr}
\usepackage{graphicx}
\usepackage{float}
\usepackage{amsthm}
\usepackage{setspace}
\usepackage{caption}
\usepackage{amsmath}
\usepackage{algorithmicx}
\usepackage{algorithm}
\usepackage{algpseudocode}
\usepackage{amssymb}
\usepackage{tikz}
\def\T{{\cal T}}

\def\Z{{\mathbb Z}}

\newtheorem{lemma}{Lemma}
\newtheorem{conjecture}{Conjecture}

\begin{document}
\phd

\title{Several topics in Experimental Mathematics}
\author{Andrew Lohr}
\program{Mathematics}
\director{Dr. Zeilberger}
\approvals{4}
\submissionyear{2018}
\submissionmonth{May}

\abstract{

This thesis deals with applications of experimental mathematics to a number of problems. The first problem is related to random graph statistics.We consider a certain class of Galton-Watson random trees and look at the total height statistic. We provide an automated procedure for computing values of the moments of this statistic. Taking limits, we confirm via elementary methods that the limiting (scaled) distributions are all the same.

Next, we investigate several problems related to lattice paths staying below a line of rational slope. These results are largely data-based. Using the generated data, we are able to find recurrences for the number of such paths for the cases of slopes 3/2 and 5/2. There is also investigation of a generalization of these problems to three dimensions.

We also examine generalizations of Sister Celine's method and Gosper's algorithm for evaluating summations. For both, we greatly extend the classes of applicable functions. For the generalization of Sister Celine's method, we allow summations of arbitrary products of hypergeometric terms and linear recurrent sequences with rational coefficients. For the extension of Gosper's algorithm, we extend it from solely hypergeometric sequences to any multi-basic sequence. For both, we have numerous applications to proving, or reproving in an automated way, interesting combinatorial problems.

We also show a partial result related to the bunk bed conjecture, a problem concerning random finite graphs. Let $G$ be a finite graph. Remove edges from $G\square K_2$ independently and with the same probability. In $G\square K_2$, there is an edge placed between all vertices of $G$ and the corresponding vertex in a copy of $G$. Then, label these vertices as either $(v,0)$ or $(v,1)$ for each $v\in V(G)$. The conjecture says that for any $x,y \in V(G)$, it is least as likely to have $(x,0)$ connected to $(y,0)$ as to have $(x,0)$ connected to $(y,1)$. We prove the conjecture in the case that only two of the edges going between the two copes of $G$ are retained.}

\beforepreface

\acknowledgements{

First, I would like to thank my advisor, Dr. Doron Zeilberger, for all of the help that he has provided as I've navigated research in graduate school. He suggested several topics for me to consider, and has always been a great advocate. Coming to Rutgers, I knew that I wanted to do combinatorics, but it took a while to settle on experimental mathematics, and he really sold me on the usefulness of the experimental approach to mathematics. 

I would like to thank my girlfriend and sometimes collaborator, Michelle Bodnar, for all the support that she has offered me through graduate school.

I would like to thank Dr. Michael Saks, not just for helping lay a groundwork of combinatorics coursework knowledge, but also for being supportive during tough times in my third year here at Rutgers.

For the topic of the final chapter of this thesis, I would like to thank Dr. Jeffry Kahn for introducing me to the bunk bed conjecture. I would also like to thank Michelle Bodnar for her advice on how to more clearly explain several of the concepts.

Of course, I would like to thank my parents for their support and encouragement to pursue mathematics.

Lastly, I would like to express my appreciation for all the other combinatorics graduate students here in the graduate program at Rutgers. They have exposed me to many topics in combinatorics I might not of otherwise known about.}

\tablestrue

\figurestrue

\afterpreface

\chapter{Introduction}

The main characteristic of what makes something experimental mathematics is the involvement of a computer at a deep level into discovering new mathematics. This often involves taking some higher level mathematical concepts and rephrasing them in such a way that questions of interest are instead mechanical computations. In this thesis, the basic mechanical computations that we will be reducing the questions of interest to are taking partial derivatives of expressions and solving systems of linear equations. Of course these are, themselves, not trivial operations, but they are tasks that have already been taught to computers.  By rephrasing out problems in terms of things that can be automated, we're able to leverage the fact that computers are faster, cheaper, and less prone to errors when trying to solve the original problems.  In chapter \ref{THS} we are able to convert a combinatorial problem involving trees into a mechanical problem involving multi-variable Calculus and symbolic computation. Another use for computers that shows up in this thesis is to compute many, many values of some sequence, and then perform statistical calculations on this data in order to suggest possible conjectures for future work. This is primarily the content of chapter~\ref{RSP}. That chapter's results are mainly conjectures that are noticed in computed data. In comparison, the other chapters are more able to give techniques for machine made proofs. In chapters \ref{Celine} and \ref{Gosper} we will address problems in summation.

Throughout, there will be some sort of computer based handling of integer sequences in order to analyze their behavior. Recurrence relations will be showing up in abundance. The kinds of recurrences that we will be considering are sometimes called P-recursive recurrences, and they are finite order linear recurrences with rational functions as coefficients. That is, we will have some quantity, either an integer sequence or an expression sequence $x_n$, and will show that it satisfies some 
\[
\sum_{j=0}^N (Q(n)R^j)x_n = 0
\]
where $R$ is the so called ``shift operator'' which is to say that for any sequence $x_n$, $Rx_n$ stands for $x_{n+1}$. We then call $N$ the order of the recurrence. We will sometimes call this whole expression the recurrence, and sometimes we will refer to only $\sum_{j=0}^N (Q(n)R^j)$ as the recurrence that $x_n$ satisfies.

There is a wealth of information on how to analyze a sequence once it is known to satisfy a particular recurrence. So, for our cases, if we can analyze a sequence to the point that we know some such recurrence that is satisfies, we will consider it solved. Since we are often starting summation problems with some undetermined number of terms that is allowed to grow arbitrarily large, any time that such a nice finite description exists, it is a cause for joy. It is not always the case that such a recurrence will exist for all possible sequences that we will consider. Since we greatly expand the classes of problems considered, we lose the completeness results enjoyed by older techniques. All of the recurrences that we consider will be of this form that they can be represented as some polynomial in $n$ and the shift operator times the sequence is equal to zero. There is a large variety of other, far more complicated kinds of recurrences that could define the sequence in question, like nested recurrences. However, not only are these more complicated expressions harder to find, but we are also able to determine less about the sequence automatically once we discover those recurrences. Indeed there are whole papers dedicated to analyzing the behavior of sequences defined by nested recurrences, so we are not so much ``solving'' the problem of analyzing the sequence by finding some way of finding such a recurrence, even though it is also a compact way of describing the sequence.

In chapter \ref{BBed},  we have used computation to obtain numerical information to bolster our confidence before attempting (and eventually finding) a proof. In addition, computation has helped provide insight in how to break down a complicated probability distribution through conditioning on different events. This verification-as-you-go approach has helped steer the direction of the proof, as many reasonable things that could be said about the problem turn out to not actually be true. Being able to rule out related problems by having a computer search for small counter examples has helped immensely. This chapter is the least experimental in its results, instead having proofs worked out by hand. This chapter's results may be able to be expanded, as the proof technique suggests that automatable proofs may be possible for related special cases of the original bunk bed conjecture. There is a purely experimental component though, in that the conjecture was verified to be true for all graphs on six vertices. To each graph, independently randomly removing each edge from $G\square K_2$ also introduces a polynomial in $p$ representing the difference of probabilities in question in the conjecture. This polynomial may be of interest independent of the truth of this conjecture.

 All of the code used for the results here, as well as results of the computations that are too bulky to fit in this document can be found at \newline{\tt http://sites.math.rutgers.edu/\~{}ajl213/DrZ/}.

The code itself  has documentation describing how to use it. These topics are rich with open problems, which will be mentioned at the end of each of their respective chapters.

\chapter{Limiting Total Height Distributions for Galton Watson Trees}
\label{THS}

The results of this chapter have been accepted for publication, and appear in \cite{Lohr1}.

\section{Background} 

While many natural families of combinatorial random variables, $X_n$, indexed by a positive integer $n$,
(for example, tossing a coin $n$ times and noting the number of heads, or counting the number of occurrences of a specific pattern
in an $n$-permutation) have  different expectations, $\mu_n$, and different standard deviations, $\sigma_n$, and (usually) largely different
asymptotic expressions for these, yet the centralized and scaled versions, $Z_n:=\frac{X_n -\mu_n}{\sigma_n}$, very often,
converge (in distribution) to the standard normal distribution whose probability density function is famously
$\frac{1}{\sqrt{2 \pi}} exp(-\frac{x^2}{2})$, and whose moments are $0,1,0,3,0,5,0,15,0,105, \dots$.
Such sequences of random variables are called {\em asymptotically normal}.
Whenever this is {\bf not} the case, it is a cause for excitement
[Of course, excitement is in the eyes of the beholder].
One celebrated case
(see \cite{Romik} for an engaging and detailed description) is the random variable
`largest increasing subsequence', defined on the set of permutations, where the intriguing {\it Tracy-Widom distribution} shows up.

Other, more recent, examples of {\it abnormal} limiting distributions are described in \cite{ZStatistics}, \cite{EkhadMoment},\cite{EkhadOEIS}, and \cite{EkhadAbnormality}.

In this chapter we consider, from an elementary, explicit, {\it symbolic-computational}, viewpoint, the
random variable `sum of distances to the root', defined over an {\it arbitrary} family of ordered rooted trees
defined by degree restrictions. For analysis of this statistic over uniformly chosen random rooted trees, see \cite{Takacs1} and \cite{Takacs2}.
The asymptotic behavior of this statistic for that uniform distribution of random rooted trees is given in \cite{Takacs3}.

It turns out that the families of trees considered in this paper are
special cases of Galton-Watson trees. These have been studied extensively by continuous probability theorists for many years,
 with a nice, comprehensive introduction given by Janson in \cite{Janson3}. For an analysis of unlabeled Galton-Watson trees, see the work Wagner\cite{Wagner}.
In particular, they are trees that are determined by determining the number of children that every node has by 
independently sampling some fixed distribution with expected value at most 1. Like the trees
considered here (described below), they are also types of Galton-Watson trees. It was shown in a three part sequence of papers by Aldous
(\cite{Aldous0}, \cite{Aldous1}, \cite{Aldous2}) and later by Marckert and Mokkadem (\cite{MaeckertMokkadem}) that all Galton-Watson generated from a 
finite variance distribution of vertex degrees followed the same distribution as 
the area under a Brownian excursion, also a topic well studied in advanced probability theory.
In particular, Janson, in section 14 of \cite{Janson1}, presents a complicated infinite sum which converges to this distribution originally discovered by Darling (1983).
Asymptotic analysis of mean, variance, and higher moments for Galton-Watson trees can be found in \cite{Janson4}.

All these authors used {\it continuous}, advanced, probability theory, that while very powerful, only gives the
limit. We are interested in {\it explicit} expressions for the first few moments themselves, or failing this,
for explicit expressions for the generating functions, for {\it any} family of rooted ordered trees given by
degree restrictions. In particular, we study in detail the case of {\it complete binary trees}, famously counted
by the Catalan numbers.

We proceed in the same vein as in \cite{EkhadOEIS}.
In that article, the random variable `sum of the distances from the root', 
defined on the set of {\it labelled rooted trees} on $n$ vertices,
was considered, and it was shown how to find explicit expressions for any given moment, and the
first $12$ moments were derived, extending the pioneering work of John Riordan and Neil Sloane (\cite{SloanesLove}),
who derived an explicit formula for the expectation. 
The exact and approximate values for the {\it limits}, as $n \rightarrow \infty$, of
$\alpha_3$ (the {\it skewness}),  $\alpha_4$ (the {\it kurtosis}),  and the higher moments
through the ninth turn out to be as follows.

$$
\alpha_3 \, = \,
{\frac{ \left( 6\,\pi -{\frac {75}{4}} \right) \sqrt {3}\sqrt {{\frac {\pi }{10-3\,\pi }}}}{10-3\,\pi }}
\, =\,
 0.7005665293596503\dots \quad , 
$$
$$
\alpha_4 \, = \, 
{\frac{-189\,{\pi }^{2}+315\,\pi +884}{ 7 \, \left( 10-3\,\pi  \right) ^{2}}}
\, = \,
3.560394897132889\dots 
\quad , 
$$
$$
\alpha_5 \, \, = \, 
{\frac { \left( 36\,{\pi }^{2}+{\frac {75}{2}}\,\pi -{\frac {105845}{224}} \right) \sqrt {3}\sqrt {{\frac {\pi }{10-3\,\pi }}}}{ \left( 10-3\,\pi 
 \right) ^{2}}}
\, =\, 
 7.2563753582799571\dots \quad, 
$$
$$
\alpha_6 \, \, = \, 
{\frac{15}{16016}}\,{\frac {-144144\,{\pi }^{3}-720720\,{\pi }^{2}+3013725\,\pi +2120320}{ \left( 10-3\,\pi  \right) ^{3}}}
\, = \, 
 27.685525695770609\dots \quad, 
$$
$$
\alpha_7 \, \, = \, 
{\frac{ \left( 162\,{\pi }^{3}+{\frac {6615}{4}}\,{\pi }^{2}-{\frac {103965}{32}}\,\pi -{\frac {101897475}{9152}} \right) \sqrt {3}\sqrt {{\frac {\pi }{
10-3\,\pi }}}}{ \left( 10-3\,\pi  \right) ^{3}}}
\, = \, 
90.0171829093603301\dots \quad, 
$$
$$
\alpha_8 \, \, = \, 
3\,{\frac {-488864376\,{\pi }^{4}-8147739600\,{\pi }^{3}-455885430\,{\pi }^{2}+86568885375\,\pi +32820007040}{ 2586584\left( 10-3\,\pi 
 \right) ^{4}}} 
$$
$$
\, = \, 358.80904151261251\dots \quad ,
$$
$$
\alpha_9 \, \, = \, 
{\frac{ \left( 648\,{\pi }^{4}+15795\,{\pi }^{3}+{\frac {591867}{16}}\,{\pi }^{2}-{\frac {461286225}{2288}}\,\pi -{\frac {188411947088175}{662165504}}
 \right) \sqrt {3}\sqrt {{\frac {\pi }{10-3\,\pi }}}}{ \left( 10-3\,\pi  \right) ^{4}}}
$$
$$
\, = \, 1460.7011342971821\dots \quad .
$$

[Note that when the moments are not centralized, the expressions are simpler, but we prefer it this way].

\section{Overview}

In this chapter we extend the work of \cite{EkhadOEIS} and treat infinitely many other families of trees.
For any given set of positive integers, $S$, we will have a `sample space' of all ordered rooted trees where
a vertex may have no children (i.e. be a {\it leaf}) or it {\bf must} have a number of children that belongs to $S$.
If $S=\{2\}$ we have the case of {\it complete binary trees}.

For each such family, defined by $S$, we will show how to derive explicit expressions for the generating functions
of the numerators of the straight moments, from which one can easily get many values, and very efficiently find the numerical
values for the moments-about-the-mean and hence the scaled moments. For the special case of complete binary
trees, we will derive explicit expressions for the first nine moments (that may be extended indefinitely),
as well as explicit expressions for the asymptotics of the scaled moments, 
and indeed (as predicted by the above-mentioned authors) they coincide {\it exactly} with those found in \cite{EkhadOEIS}
for the case of labeled rooted trees. This is a specific example of a more general statement about Galton Watson trees given in \cite{Janson4}. 

\section{Rooted Ordered Trees}

Recall that an {\it ordered rooted tree} is an unlabeled graph with the root drawn at the top, and each
vertex has a certain number (possibly zero) of children, drawn from left to right. For any finite set
of positive integers, $S$, let  $\T(S)$ be the set of all rooted labelled trees where each vertex either has
no children, or else has a number of children that belongs to $S$. The set $\T(S)$ has the
following structure (``grammar'')
$$
\T(S) = \{\cdot\} \bigcup_{i \in S}  \, \{\cdot\} \times \T(S)^i \quad.
$$

Fix $S$, Let $f_n$ be number of rooted ordered trees in $\T(S)$ with exactly $n$ vertices.
It follows immediately, by elementary generatingfunctionology, that
the ordinary generating function
$$
f(x) :=\sum_{n=0}^{\infty} f_n  \, x^n \quad ,
$$
(that is the sum of the weights of {\it all} members of $\T(S)$ with the weight $x^{NumberOfVertices}$ assigned to each tree)
satisfies the {\bf algebraic} equation
$$
f(x) = x \left ( 1+ \sum_{i \in S} f(x)^i \right ) \quad .
$$

Given an ordered tree, $t$, define the random variable $H(t)$ to be the sum of the distances to the root of all vertices.
Let $H_n$ be its restriction to the subset of $\T(S)$, let us call it $\T_n(S)$, of members of $\T(S)$ with exactly
$n$ vertices. Our goal in this chapter is to describe a symbolic-computational algorithm that, for {\it any}
finite set $S$ of positive integers, {\it automatically} finds  generating functions that enable the fast
computation of the average, variance, and as many higher moments as desired. We will be particularly interested
in the limit, as $n \rightarrow \infty$, of the centralized-scaled distribution, and we confirm that it is always the same as the one for rooted labelled trees found in \cite{EkhadOEIS} as we would expect by \cite{Janson4}.

Let $P_n(y)$ be the generating polynomial defined over $\T_n(S)$, of the random variable, `sum of distances from the root'.
Define the {\it grand generating function}
$$
F(x,y)=\sum_{n=0}^{\infty} P_n(y) x^n \quad .
$$

Consider a typical tree, $t$, in $\T_n(S)$, and now define the more general {\it weight} by $x^{NumberOfVertices}\, y^{H(t)}=x^n \, y^{H(t)}$.
If $t$ is a singleton, then its weight is simply $x^1 y^0=x$, but if its sub-trees (the trees whose roots are the children
of the original root) are $t_1, t_2, \dots t_i$ (where $i \in S$), then 
$$
H(t)=H(t_1)+ \dots + H(t_i) + n-1 \quad,
$$
since when you make the tree $t$, out of  subtrees $t_1, \dots, t_i$ by placing them from left to right and then
attaching them to the root, each vertex gets its `distance to the root' increased by $1$, so altogether the 
sum of the vertices' heights gets increased
by the total number of vertices in $t_1, \dots, t_i$ (i.e. $n-1$). Hence $F(x,y)$ satisfies the
{\bf functional equation}
$$
F(x,y)=x \cdot  \left ( 1+ \sum_{i \in S} F(xy, y)^i \right ) \quad,
$$
that can be used to generate many terms of the sequence of generating polynomials $\{ P_n(y) \}$.

Note that when $y=1$, $F(x,1)=f(x)$, and we get back the algebraic equation satisfied by $f(x)$.

\section{From Enumeration to Statistics in General}

Suppose that we have a finite set, $A$, on which a certain numerical attribute, called {\it random variable}, $X$,
(using the language of probability and statistics) is defined.

For any non-negative integer $i$, let us define
$$
N_i:=\sum_{a \in A} X(a)^i \quad .
$$
In particular, $N_0(X)$ is the number of elements of $A$.

The expectation of $X$, $E[X]$, denoted by $\mu$  is, of course,
$$
 \mu  \, = \, \frac{N_1}{N_0} \quad .
$$

For $i>1$, the $i$-th straight moment is
$$
E[X^i] \, = \, \frac{N_i}{N_0} \quad .
$$

The $i$-th {\it moment about the mean} is
$$
m_i:=E[(X-\mu)^i]= E[\sum_{r=0}^{i} {{i} \choose {r}} (-1)^r \mu^r X^{i-r}]=
\sum_{r=0}^{i}  (-1)^r {{i} \choose {r}} \mu^r E[X^{i-r}]
$$
$$
=\, \sum_{r=0}^{i}  (-1)^r {{i} \choose {r}} \left ( \frac{N_1}{N_0} \right )^r   \frac{N_{i-r}}{N_0}
$$
$$
= \, \frac{1}{N_0^i} \sum_{r=0}^{i}  (-1)^r {{i} \choose {r}} N_1^r N_0^{i-r-1} N_{i-r}  \quad .
$$

{\bf Finally}, the most interesting quantities, statistically speaking, apart from the mean $\mu$ and variance $m_2$ are
the {\bf scaled-moments}, also known as, {\it alpha coefficients}, defined by
$$
\alpha_i :=\frac{m_i}{m_2^{i/2}} \quad .
$$

\section{Using Generating Functions}

In our case $X$ is $H_n$ (the sum of the vertices' distances to the root, defined over rooted ordered trees in our family,
with $n$ vertices), and we have
$$
N_1(n) = P_n'(1)
$$
$$
N_i(n) = (y\frac{d}{dy})^i P_n(y) \bigl \vert_{y=1} .
$$
It is more convenient to first find the numerators of the factorial moments
$$
F_i(n)=(\frac{d}{dy})^i P_n(y) \vert_{y=1} \quad,
$$
from which $N_i(n)$ can be easily found, using the Stirling numbers of the second kind.

\section{ Automatic Generation of Generating Functions for the (Numerators of the) Factorial Moments}

Let us define
$$
P(X)=1+ \sum_{i \in S} X^i \quad,
$$
then our functional equation for the grand-generating function, $F(x,y)$ can be written
$$
F(x,y)=xP(F(xy,y) )\quad .
$$
If we want to get generating functions for the first $k$ factorial moments of our random variable $H_n$, we need
the first $k$ coefficients of the Taylor expansion, about $y=1$, of $F(x,y)$. Writing $y=1+z$, and
$$
G(x,z)=F(x,1+z) \quad,
$$
we get the functional equation for $G(x,z)$
$$
G(x,z)=x\, P( G(x+xz,z))  \quad .
\eqno(FE)
$$
Let us write the Taylor expansion of $G(x,z)$ around $z=0$ to order $k$
$$
G(x,z)=\sum_{r=0}^{k} g_r(x) \frac{z^r}{r!} + O(z^{k+1}) \quad.
$$
It follows that
$$
G(x+xz,z)=\sum_{r=0}^{k} g_r(x+xz) \frac{z^r}{r!} + O(z^{k+1}) \quad.
$$

We now do the Taylor expansion of $g_r(x+xz)$ around $x$, getting
$$
g_r(x+xz)=g_r(x) \, + \, g'_r(x)(xz) \, +  \, g''_r(x) \frac{(xz)^2}{2!}  \, + \, \dots \, + \, g_r^{(k)}(x) \frac{(xz)^k}{k!} \, + \, O(z^{k+1}) \quad .
$$ 

Plugging all this into $(FE)$, and comparing coefficients of respective terms of $z^r$ for $r$ from $0$ to $k$
we get $k+1$ equations relating $g^{(j)}_r(x)$ to each other. It is easy to see that
one can express $g_r(x)$ in terms of $g_{s}^{(j)}(x)$ with $s<r$  (and $0 \leq j \leq k$) .

Using implicit differentiation, the derivatives of $g_0(x)$, $g_0^{(j)}(x)$ (where $g_0(x)$ is the same as $f(x)$), can be expressed
as rational functions of $x$ and $g_0(x)$. 
As soon as we get an expression for $g_r(x)$ in terms of $x$ and $g_0(x)$, we can
use calculus to get expressions for the derivatives $g_r^{(j)}(x)$ in terms of $x$ and $g_0(x)$. At the end of the day,
we get expressions for each $g_r(x)$ in terms of $x$ and $g_0(x)$ (alias $f(x)$), and since it is easy to find the first
ten thousand (or whatever) Taylor coefficients of $g_0(x)$, we can get the first ten thousand coefficients of
$g_r(x)$, for all $0 \leq r \leq k$, and get the numerical sequences that will enable us to get very good approximations for the alpha coefficients.

The beauty is that this is all done by the computer! Maple knows calculus.

We can do even better. Using the methods described in \cite{FlajoletSedgewick}, one should be able to get, {\it automatically},
asymptotic formulas for the expectation, variance, and as many moments as desired. Using these techniques, it may be possible to obtain expressions for the leading terms of all moments, and so show weak convergence of this distribution to a particular limiting distribution. This should be an interesting future project.

For the special case of complete binary trees, everything can be expressed in terms of Catalan numbers, and
hence the asymptotic is easy. For more general $S$ sets, we do not have the same beatutiful formula that we get for the binary case, but we can still give information about the asymptotics. Our beloved computer, running the Maple package {\tt TREES.txt} (mentioned above),
obtained the results in the next section.

{\bf Computer-Generated Theorems About the Expectation, Variance, and First Nine Moments for the Total Height on Complete Binary Trees on $n$ Leaves}

See the output file

{\tt http://www.math.rutgers.edu/\~{}zeilberg/tokhniot/oTREES3.txt} .

\section{Universality}

The computer output, given in the above webpage, proved that for this case, of complete binary trees, the  limits of the first nine scaled moments coincide
{\it exactly} with those found in \cite{EkhadOEIS}, and given above. 
This confirms, by {\it purely elementary, finitistic methods}, the universality property mentioned above.
We do it for one family at a time, and only for finitely many moments, but on the other hand, we derived
{\it explicit} expressions for the first twelve moments in the case of complete binary trees, and explicit
expressions for the generating functions for the moments for other families.

\chapter{Rational Sloped Paths}
\label{RSP}
\section{Background}

There is a rich study of Dyck paths in combinatorics. Some of the most ubiquitous results are for the case that the slope of the line is 1. In particular that the number of paths from (0,0) to (n,n) is counted by the Catalan numbers. When we change it from 1 to another rational number $a/b$, we enter the realm of appropriately named rational Catalan combinatorics. For notational convenience, we will let $A_{a,b,n}$ denote the number of paths from $(0,0)$ to $(bn,an)$ staying on or below the line $y=a/bx$.

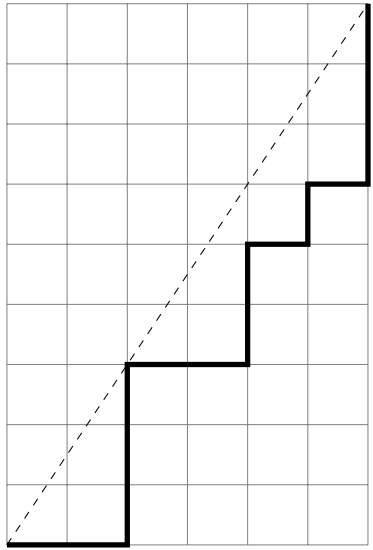
\begin{figure}[h]
\centering
\begin{tikzpicture}[scale=0.8]
(0,0) rectangle +(6,9);
\draw[help lines] (0,0) grid +(6,9);
\draw[dashed] (0,0) -- +(6,9);
\coordinate (prev) at (0,0);
\draw [ line width=2] (0,0)--(2,0)--(2,3)--(4,3)--(4,5)--(5,5)--(5,6)--(6,6)--(6,9);
\end{tikzpicture}
\caption{One of the paths counted by $A_{3,2,3}$}
\label{vertRun}
\end{figure}.

It was shown by Duchon in 2000 that for any slope $\frac{a}{b}$, the number of paths below a line of that slope is asymptotically $\Theta\left(\frac{1}{n} \binom{(a+b)n}{an}\right)$ \cite{Duchon}.  However, it is still unknown what the constant out front is. To show the asymptotics, Duchon showed that the content is somewhere between $\frac{1}{a+b}$ and  $\frac{1}{a}$. This upper bound on the number of paths was known at least as far back as 1950 to Grossman. Grossman also had an interesting result, the first proof of which is given by Bizley in 1954 in a now defunct actuarial journal \cite{Bizley}. It gives that gives an exact formula for every $A_{a,b,n}$. Of course, this precision comes at a cost, The formula is given as a sum over a large set of weighted integer partitions. There is no good way to extract estimates from this formula that we know, but it seems powerful and may be useful for this problem in the future. It would be great to have a simpler explanation of the simpler problem of determining this value up to a $(1+o(1))$ factor

It is clear that $A_{a,b,n} = A_{b,a,n}$ because you could take any valid path, rotate it by a quarter turn and take a reflection. So in this discussion, we will assume that we always have $a>b$.

 For the case $b=1$, there is an exact solution known, using Fuss-Catalan numbers $\frac{1}{1+an} \binom{(1+a)n}{an}$.

\section{Approach}

We would like to try to find the coefficient out front, that is, $\alpha$ so that the number of paths is $\left(1+o(1)\right) \frac{\alpha}{n} \binom{(a+b)n}{an}$. In an effort to do this, we first are tasked with computing many terms of the sequence

\begin{algorithm}
\caption{Dynamic programming approach to generating data}
\begin{algorithmic}[1]
\State{$a[*,0] = 1$ }
\For{$i$ from $1$ to $an$}
	\For{$j$ from 1 while $ bj \le ai$}
		\State{a[i,j] = a[i-1,j] + a[i,j-1]}
	\EndFor
	\State{forget $a[i-1,*]$}
\EndFor
\State{$A_{a,b,n} = a[an,bn]$}
\end{algorithmic}
\end{algorithm}

Through a simple dynamic programming algorithm, we are able to compute the number of paths from $(0,0)$ to $(bn,an)$ for $n$ around 1000. This gives us enough data to try and find a recurrence relation that it satisfies using the maple package available at {\tt http://www.math.rutgers.edu/~zeilberg/tokhniot/FindRec.txt} However, we were only able to successfully find recurrences for the slopes $3/2$ and $5/2$. Armed with these recurrences we are able to blindingly fast crank out many thousands more terms of this sequence. The recurrence in the data file for slope 3/2 is order 4, whereas the one for 5/2 is order 8 and monstrously long. Though we cannot guarantee this is the minimal recurrence, it still gives a massively way faster of counting the paths for these two unknown slopes, and potentially for many more slopes that our computer was not keen enough to find this time.

Then, once we have exact numbers, we do a statistical fit of the data for many values of n against the model $\left(\frac{\alpha}{n} + \frac{\beta}{n^2}+\frac{\gamma}{n^3}+\frac{\delta}{n^4}\right)\binom{(a+b)n}{an}$ to get our estimate of $\alpha$. Adding more error terms did not affect the value of $\alpha$ much. We estimate how close this is to the truth by running it for the first 100 values of n, and the first 200 values of n, and seeing how much our estimate stays the same. We can get quite good estimates of these numbers!

\section{Data and Figures}
\begin{table}
\centering
\caption{Estimates for $\alpha$ part 1}
\begin{tabular}{|c||c|c|c|c|c|c|c|c|c|c|}
\hline
a\textbackslash b&2&3\\
\hline
\hline
2&1&\\
3&0.240706636&1\\
4&0.50000000&0.15972479544\\
5&0.1613399969&0.1372518253\\
6&0.333333333&0.50000000\\
7&0.1216701970&0.1073342967\\
8&0.250000000&0.09683505915\\
\hline
\end{tabular}
\end{table}

\begin{table}
\centering
\caption{Estimates for $\alpha$ part 2}
\begin{tabular}{|c||c|c|c|c|c|c|c|c|}
\hline
a\textbackslash b&4&5&6&7\\
\hline
\hline
4&1&&&\\
5&0.119952918&1&&\\
6&0.240706636&0.09621264003&1 &\\
7&0.09639805178&0.08763172133&0.08039623916&1\\
8&0.5000000&0.08048157890&0.159724795&0.06908631788\\
\hline
\end{tabular}
\end{table}

There are many more slopes for which we have very exact estimates of $\alpha$, and are available on line at

{\tt http://www.math.rutgers.edu/\~{}ajl213/DrZ/RSP.html} in the extra data file.

Though we knew it already with Duchon's result that the value of the coefficient is at most $1/a$, it still seems surprising that the value of the coefficient is not monotone in the value of the slope, that is the actual number $a/b$. We can notice a few simple patterns here, in particular, except in the cases that $a$ and $b$ are not in lowest terms, the value of coefficient decreases as you increase either $a$ or $b$. 

We can investigate this second observation a little further. A similar, but interesting and distinct problem is to try letting something else go to infinity in $A_{a,b,n}$ other than n, as we had before.

Suppose instead that we were to let $a$ go to infinity while $b$ is fixed. There is a nice pattern that appears. In particular, we have the conjecture that $\alpha$ is asymptotically equal to $\frac{gcd(a,b)}{a}$. The $gcd(a,b)$ factor is expected because it makes the expression only depend on the value of $a/b$, as it should.

\begin{figure}[H]
\centering
\caption{$a\alpha$ as a function of a for $b=2$}
\includegraphics[height=.4\textwidth,width=.8\textwidth]{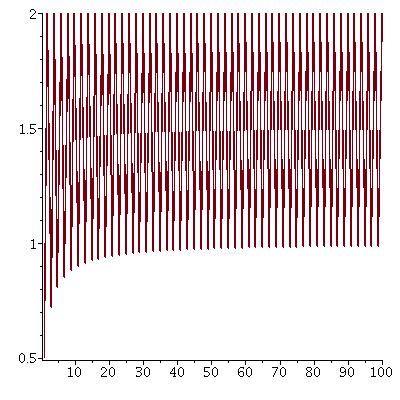}
\end{figure}

\begin{figure}[H]
\centering
\caption{$a \alpha$ as a function of a for $b=8$}
\includegraphics[height=.4\textwidth,width=.8\textwidth]{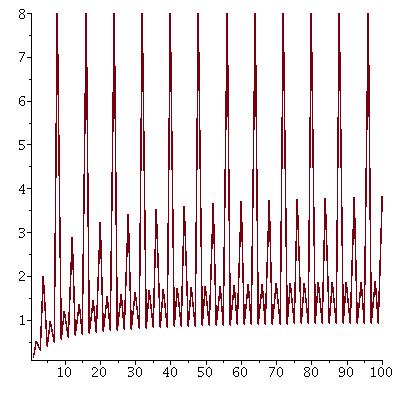}
\end{figure}

\section{Time Above The Line}
A well known result for the case of slope 1 lines is the so called ``Arc-sine law'' which concerns the time that a random walk from $(0,0)$ to $(n,n)$ spends to one side or the other of the line $y=x$. For a discussion of this principle see chapter 6 of \cite{RandProcess}, or, for a more focused discussion of this, see the lecture notes \cite{Reflection}. Follow a random lattice walk from $(0,0)$ to $(n,n)$ and after each step, you make a note of whether you are above or below the line. Then, if you look at the value of the number of times you were above minus the number of times you were below, it is very unlikely that they were roughly the same. More precisely, if you correctly normalize this difference, then asymptotically, the distribution of this difference is going towards 
\[
\frac{1}{\pi\sqrt{k(n-k)}}
.\]
So there are large spikes near all the time above the line and all the time below the line, with a very low trough in the middle corresponding to equal amounts of time above and below the line. It gets its name because its CDF given by arcsine. Even though this may suggest that there is a geometric proof of the fact, none of its proofs are geometric. There is a nice proof that the time above the line statistic follows this distribution using the reflection principle given by Sparre Andersen in \cite{Andersen}.

For our approach, we are able to compute this quantity extremely quickly using a technique involving generating functions. In Figure \ref{3_2_400}, we plot the number of paths from $(0,0)$ to $(400,600)$ with exactly $5k$ time intervals above the line for each value of $k$.
\begin{figure}
\caption{analog of arc sine law for slope 3/2}
\label{3_2_400}
\includegraphics[scale = .5]{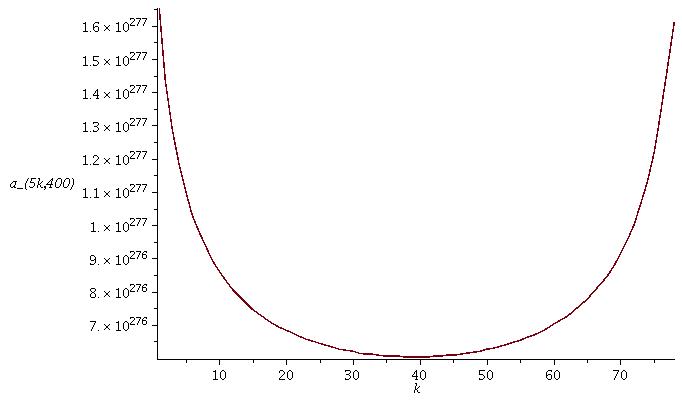}
\end{figure}

Notice that empirically, this distribution is the same as for the well studied case of a slope 1 line. However, there are some important distinctions to mention.

Since we are plotting $5k$, we are only going to be considering one congruence class of possible times spent above the axis. This is also true of the slope 1 result, in that it only considers even times above the axis. However, for the slope 1 case, this is because there are no paths with an odd amount of time above the axis. For the other congruence classes with slope $3/2$, we do get paths with that much time spent above the axis, they are just not as well behaved, usually possessing a significant skew one way or the other. This is also a reason that the reflection principle that is used to prove the arcsine result for slope 1 fails for our slope 3/2. Put another way, it fails because there are multiple different ways you could cross the line, you could cross at a point in your lattice, or in between points in the lattice. This is unlike in the slope 1 case where they all look the same, since you need to cross at a point in the lattice. See figure \ref{3_2_400_offset} for the example where we only consider paths that are length 4 mod 5. Of all the offsets, it is the most extremely skewed in favor of staying below the line.

\begin{figure}
\caption{analog of arc sine law for slope 3/2 with offset}
\label{3_2_400_offset}
\includegraphics[scale = .5]{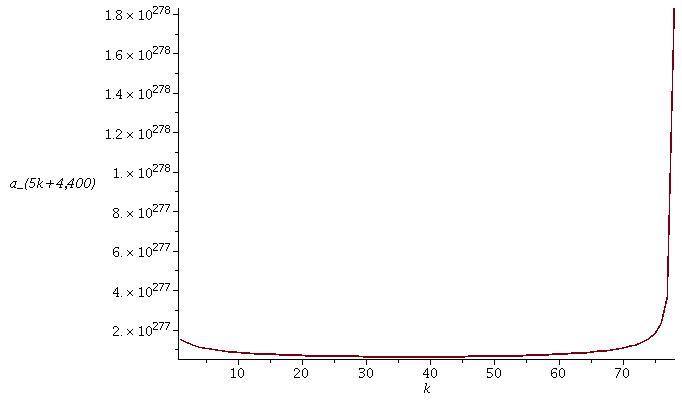}
\end{figure}

\section{Three Dimensional Lattice Walks}

A much less studied area concerns paths in a three dimensional lattice($\Z^3$) that have to stay to stay in a region bounded by planes. It is simple to extend the dynamic programming solution to this situation. However, since there many more lattice points, the runtime goes up from $\Theta(n^2)$ to $\Theta(n^3)$. This keeps us from getting anywhere near as much data as we did in the previous section. With the data we do have, we have the suggestion that something much more interesting than in the 2D case is happening!

There is a treatment of multidimensional lattice walks in a text on combinatorics by Bona \cite{Bona} where they are able to get wonderful exact results using a generating function approach. However, their problem is a bit different than the generalization that we chose in that they require that their planes look like $ d \ge ax+by+cz$ with strictly positive $a$, $b$, and $c$. They also only consider a single plane, and so it seems unlikely that the approach that they mention would apply to our problem. It is, however, a really striking result that is able to completely nail down the problem and so would likely be of interest to those tackling our problem that we present here.

The way that we set up the three dimensional problems, is that we take the number of paths with steps in $\{\langle1,0,0 \rangle,\langle0,1,0 \rangle, \langle0,0,1 \rangle\}$. Instead of the 2D problem of requiring $ax<by$, we define and instance of the counting problem to be indexed by three numbers, $a,b,c$, and require of our paths that they satisfy $ax\le by\le cz$ that end at $x=bc,y=ac,z=ab$. If we have $a=b=c=1$, this has the precise formula of the 3D Catalan numbers (A005789 in \cite{OEIS}) $ \frac{2}{(n+1)^2(n+2)} \binom{3n}{n,n,n}$. However, there is a lot left to understand in this problem, and some things that are distinctly different than the 2D case. In particular, for 2D, it was always  $\Theta\left(\frac{1}{n} \binom{(a+b)n}{an}\right)$. That is, the slope of the line did not affect the fact that you always had a $\Theta(\frac{1}{n})$ fraction of all paths. For the already known $a=b=c=1$, it is a $\Theta(\frac{1}{n^3})$ fraction of all paths, however this appears to change for different choices of $a,b,c$. We have some data on the value of this coefficient in tables \ref{tab:RSP3D1} and \ref{tab:RSP3D2}.

Our problem is in a way very related to the multi-candidate ballot problem presented in \cite{ZGessel} which is phrased in terms of Weyl chambers. For their technique, they are instead requiring that $a=b=c=1$, so, looking at higher dimensional generalizations (of what is already known in this three dimensional problem).

\begin{table}
\centering
\caption{\label{tab:RSP3D1}a=1}
\begin{tabular}{|c||c|c|c|c|c|c|c|}
\hline
b\textbackslash c&1&2&3&4&5&6&7\\
\hline
\hline
1&3.0&2.7&2.6&2.5&2.5&2.4&2.4\\
2&3.7&3.3&3.0&2.9&2.8&2.7&2.6\\
3&4.3&3.7&3.4&3.2&3.1&3.0&2.9\\
4&4.8&4.1&3.8&3.5&3.4&3.2&3.1\\
5&5.2&4.5&4.1&3.8&3.6&3.4&3.3\\
6&5.7&4.8&4.4&4.1&3.8&3.6&3.5\\
7&6.0&5.2&4.6&4.3&4.0&3.8&3.7\\
\hline
\end{tabular}
\end{table}

\begin{table}
\centering
\caption{\label{tab:RSP3D2}a=2}
\begin{tabular}{|c||c|c|c|c|c|c|c|}
\hline
b\textbackslash c&1&2&3&4&5&6&7\\
\hline
\hline
1&2.7&2.5&2.5&2.4&2.4&2.3&2.3\\
2&3.3&3.0&2.8&2.7&2.6&2.6&2.5\\
3&3.7&3.4&3.2&3.0&2.9&2.8&2.8\\
4&4.1&3.7&3.5&3.3&3.1&3.0&3.0\\
5&4.5&4.0&3.7&3.5&3.4&3.2&3.2\\
6&4.8&4.3&4.0&3.7&3.6&3.4&3.3\\
7&5.2&4.6&4.2&4.0&3.8&3.6&3.5\\
\hline
\end{tabular}
\end{table}

By fitting this data (with many more digits of each number than we have listed in this table), we are able to get conjecturally that

\[
\alpha = 2 + \frac{1}{\sqrt{c}}.
\]

This is really bizarre that not only are we having fractional powers of $n$ of the path, but irrational powers of $n$ showing up. This would also suggest that applications of the techniques of chapter \ref{Celine} would not be useful here as the Birkhoff-Tirijinski method \cite{WZ} for finding dominant asymptotics is unable to yield irrational powers of $n$ in its dominant asymptotics. 

For analyzing these paths, it is often helpful to view them as words instead. They are words in the three basic steps that are allowed to be taken, either a unit step in the positive $x$ direction, in the positive $y$ direction, or in the positive $z$ direction, denoted by the letters $x$, $y$, and $z$ respectively. So, for example, if we have $a=2$, $b=1$, $c=3$ and $n=1$ we could have $zyyxzyyyxy$ from among the 54 possible paths with that choice of $a$, $b$, and $c$. 

From these tables, one simple lower bound that one might notice is that none of these coefficients is greater than two. Before we get to proving this fact, we introduce a definition that will be useful. We say that two words $w_1$ and $w_2$ are {\em y-conjugated} if we can factor them in such a way that there is a $v_1$ and $v_2$, such that $w_1 = v_1 y v_2$ and $w_2 = v_2 y v_1$. Another way of viewing this is that we can get the the y-conjugates of a word $w$ by attaching a $y$ to the end, performing some number of cyclic shifts of this word until a $y$ is once again at the end, and finally peeling of the trailing $y$. This second interpretation makes it much clearer that this is an equivalence relation. Motivated by this standard definition of conjugation about a letter, we define a new kind of conjugation about a pair of letters. 

Define $xz$-conjugation by first picking two particular occurrences of an $x$ and a $y$ and then split up the word $w$ into $w = v_1 x v_2 y v_3$. Define the $xy$ conjugate to be $w' = v_3 y v_2 x v_1$. If our occurrence of $y$ was first, then that is not an issue, define the conjugate of $ v_1 y v_2 x v_3$ to be $v_3 x v_2 y v_1$. We show that this is also an equivalence relation. It is clearly symmetric since the operation is its own inverse, picking the same occurrences of $x$ and $y$.

Further analyzing these sequences seems like a fruitful place to make progress, as very little is known about their growth.

\chapter{Generalization of Sister Celine's Method}
\label{Celine}

The results of this chapter have been accepted for publication, and appear in \cite{Lohr2}.

\section{Background}

One of the earliest steps in automatically proving identities dates back to Sister Mary Celine Fasenmyer's 1945 Ph.D. thesis \cite{Fasenmyer1}. She gave a technique for computing sums of proper hypergeometric terms, also see \cite{Fasenmyer2}. Her technique concerns sequences of the form  $x_n = \sum_k H(n,k)$, where the sum is over all $k$ so that $H(n,k)$ is non-zero. Because it is summing over all of these $k$, the problems that it can be applied to only make sense if for each $n$ there are only finitely many values of $k$ that cause $H(n,k)$ to be non-zero. Many expressions constructed from binomial coefficients fit these requirements. It also requires that $H(n,k)$ is proper hypergeometric, meaning that it can be written in the form:
\[
F(n,k) = P(n,k) \frac{\prod_{i=0}^{U} (a_in+b_ik+c_i)!}{\prod_{i=0}^{V} (u_in+v_ik+w_i)!} x^k
\].
For some finite $U,V$,  $i$, $a_i,b_i,c_i,u_i,v_i,w_i \in \Z$, $x$ an indeterminate, and $P$ a polynomial. A simpler way to phrase this is that $H$ needs to be a polynomial times some rational expression of factorials. One implication of this is that both $\frac{H(n,k+1)}{H(n,k)}$ and $\frac{H(n+1,k)}{H(n,k)}$ are rational functions in $n$ and $k$. In order to determine if there is an order $I$ recurrence for the sequence, her technique picks some $J$ and considers 

\[
0 = \sum_{i=0}^I \sum_{j=0}^J y_{i,j}(n) H(n+i,k+j),
\]
where $y_{i,j}(n)$ is an as yet unknown rational function of $n$. If the value picked for $J$ was not large enough then this procedure will fail, and a higher value for $J$ would be considered. Then, by $H$ being hypergeometric, it is able to reduce all of the $H(n+i,k+j) = G_{i,j}(n,k)H(n,k)$ where $G_{i,j}$ is some rational function of $n$ and $k$. From there, divide everything through by $H(n,k)$. Now, we have something of the form 

\[
0 = \sum_{i=0}^I \sum_{j=0}^J G_{i,j} (n,k) y_{i,j}(n).
\]

Combining denominators on the right hand side, and multiplying through by the common denominator, we get that the right hand side becomes a polynomial in $n$,$k$, and $\{y_{i,j}(n)\}$. Collect terms by what power of $k$ appears, and then solve for what the $\{y_{i,j}(n)\}$ have to be in order to make all of the coefficients of powers of $k$ equal to zero. We may get unlucky and have no solution. Then, we would need to try a larger $I$ and $J$ to begin with. If however, we find a solution, we plug that into where we first introduced $y_{i,j}(n)$. Since these have no $k$'s in them, and $x_n$ is obtained by summing over all values of $k$ that make the summand nonzero, we have

\[
0 = \sum_{i=0}^I \sum_{j=0}^J y_{i,j}(n) H(n+i,k+j) = \sum_{i=0}^I \left( \sum_{j=0}^J y_{i,j}(n)\right) x_{n+i}.
\]

Which we may write in shift operator notation as 
\[
0 = \left(\sum_i^I \left( \sum_{j=0}^J y_{i,j}(n)\right) N^i\right)x_n.
\]

At this point we say that we are done. First, having a recurrence allows us to compute the sequence out to very large values very quickly, storing only a constant number of terms. Also, once we have a rational recurrence like this for $x_n$ then we can extract as good asymptotics as desired like using techniques by Birkhoff-Trjizinski which has been nicely summarized in \cite{WZ}. Sometimes, but not always, we are able to recover a really nice formula such as a product of rational functions, factorials, and binomial coefficients. Some of the recurrences found by our procedure are very complicated, so there is little hope to always be able to recover a formula.

For a more complete explanation of Sister Celine's method, look at chapter 4 of \cite{AeqB}. There are some generalizations of Sister Celine's method given in \cite{ZCeline}, in particular to certain classes of multiple summations and to a continuous analog.

Some of our applications of the expanded method presented in this paper relate to binomial transforms of functions. There are nice treatments of binomial transforms of Fibonacci like sequences given in \cite{Spivey}.

\section{Overview}

We take the described technique of Sister Celine and extend it to allow many more kinds of summands. In particular, the sequence can be of the form $x_n =\sum_k^n a_k^d H(k,n)$ where d is any number,  H is hypergeometric, and $a_k$ is some sequence defined by a rational recurrence relation. Since so many sequences can be so described by rational recurrence relations, this is a significant extension in scope. 

It works very similarly to Sister Celine's method, in that we will consider ratios of successive terms. That is,  to find a recurrence with order at most $I$, start with

\[
\sum_{i=0}^I \sum_{j=0}^{J} \frac{H(n+i,k+j)}{H(n,k)} a_{j+k}^d y_{i,j}(n).
\]

Let $D$ be the order of the recurrence describing $\{a_k\}$. Then, we use that relation to rewrite all of the $\{a_{k+j}\}_{j=D}^J$ in terms of $\{a_{k+j}\}_{j=0}^{D-1}$. That is, by repeatedly applying the relation, we can write each $a_{j+k}$ as a linear combination:

\[
a_{j+k}  = \sum_{m=0}^{D-1} c_{k,j,m} a_{k+m},
\]
where for the $j<D$, we just let 
\[
c_{k,j,m} = \begin {cases} 1 & j=m\\ 0 & j \neq m \end{cases}.
\]

Then, since we have an expression with $D$ terms to the $d$, we can expand that out to get at most $Dd$ terms. Then, unlike in Sister Celine's method, where we have a polynomial in $k$, we now have a polynomial in $\{k,a_k,a_{k+1},\ldots a_{k+D-1}\}$. But, once we have collected the coefficients of each of the combinations of those variables, we set all of them equal to zero, and then try to solve for the $y_{i,j}(n)$. Since we are able to keep the number of variables that it is a polynomial in bounded by $D$, we are able to be sure that the number of equations won't exceed $\binom{D+|J|}{|J|}$, while the number of rational functions that we are allowing us to pick is $|I||J|$. So, when we are looking for a recurrence, we can try and allow us larger and larger orders, until we have the freedom needed to find all of the rational functions needed to make the recurrence true. As in Sister Celine's method, we are not guaranteed that we can find such a solution for our particular choice of $I$ and $J$. We are guaranteed by WZ theory \cite{AeqB} that for a large enough choice of $I$ and $J$, it gives us a recurrence relation that looks like

\[
 0 = \left( \sum_{i=0}^I \left(\sum_{j=0}^J y_{i,j}(n)\right) N^i \right) x_n.
\]

Our whole technique is implemented in a Maple package whose address is given at the beginning of this paper. The usefulness of our technique comes from being easily carried out by a computer, since the systems of equations involved quickly get too large for a person. We invite the reader to use this package the next time that the come across a type of summation problem that they want to analyze.

\section{Application to Enumerating Chess King Walks}

Suppose that there is a king wandering around on an infinite $d$-dimensional chess board. We want to know how many of the $(3^d-1)^n$ walks of length $n$ that the king could take would end up bringing him back to where he started. Given a polynomial $p$, we use the notation $Ct(p)$ to denote the constant term of $p$. Then, by using the powers of $z_i$ to keep track of our total displacement in the $i$ dimension, we have:
\begin{align*}
x_n &=Ct\left(\left(\left(\prod_{i=1}^d z_i +z_i^{-1} +1\right) -1\right)^n\right)\\
&=Ct\left(\sum_{k=0}^n\left(\prod_{i=1}^d z_i +z_i^{-1} +1\right)^k \binom{n}{k} (-1)^{n-k}\right)\\
&=\sum_{k=0}^n Ct\left(\left(\prod_{i=1}^d z_i +z_i^{-1} +1\right)^k\right) \binom{n}{k} (-1)^{n-k}\\
&=\sum_{k=0}^n Ct\left(\left(z +z^{-1} +1\right)^k\right)^d \binom{n}{k} (-1)^{n-k}.
\end{align*}
 
Luckily for us, $Ct\left(\left(z +z^{-1} +1\right)^k\right)$ is already well understood. It is the sequence of central trinomial coefficients (A002426 \cite{OEIS}). Also luckily, it is known that this sequence satisfies the recurrence
\[
 0 = \left(N^2 - \frac{2n-1}{n}N - \frac{3n-3}{n}\right)x_n.
\]

So, we are in exactly the set up of our technique. In which case, if we let $a_k = Ct\left(\left(z +z^{-1} +1\right)^k\right)$, we can describe the number of $d$ dimensional king walks which end at the origin after taking $n$ steps by 

\[\sum_{k=0}^n a_k^d \binom{n}{k}(-1)^{n-k}.\]

Once the counting problem has been rewritten as this sum, it clearly falls into the scope of our technique. Using it we are able to find rational recurrences (effectively solve) for all dimensions up to 4.
For a two dimensional king walking around, if we let
\begin{align*}
g(n,N) =& (3n^3+40n^2+175n+250)N^3  \\
 &+(9n^3+138n^2+703n+1190)N^2\\
&+(108n^3+1548n^2+7364n+11632)N\\
&+96n^3+1280n^2+5632n+8192
\end{align*}
then
\[
0 = g(n,N)x_n.
\]
Although this recurrence already looks a little ugly, at least it is short, which is more than can be said of those expressions describing higher dimensions. But they are included in an appendix. Also important is that they were found by a computer.

Something more insightful than looking at the walls of text you see when looking at the recurrences that exactly describe these sequences is looking at their their asymptotics:

For the two dimensional king, the number of paths of length n is:
\[
c_2 \frac{8^n}{n}\left( 1 - \frac{4}{9n} + \frac{1}{18n^2} + O\left(\frac{1}{n^3}\right)\right),
\]
For three dimensions the number is:
\[
c_3\frac{26^n}{n^\frac{3}{2}}\left( 1 - \frac{11}{18n} + \frac{683}{5832n^2} + O\left(\frac{1}{n^3}\right)\right).
\]
and for four dimensions the number is:
\[
c_4 \frac{80^n}{n^2} \left(1 - \frac{25}{9n} + \frac{36439}{6561n^2}+O\left(\frac{1}{n^3}\right)\right).
\]

The dominant asymptotics are somewhat unsurprising. The exponential part is all possible paths. The dominant power of $n$ is $\left(\frac{1}{\sqrt{n}}\right)^d$. It is well known that the central binomial coefficient is asymptotically $\frac{2^n}{\sqrt{n}}$, and we are doing something somewhat like that in $d$ dimensions. The value of $c_2$ is approximately equal to $\frac{2}{3\pi}$. This value for $c_2$ can be proven in a rigorous way using classical analysis. For $c_3$ and $c_4$, we are not so lucky, instead, all we can say from non-rigorous observation is that $c_3 \approx .110225343716$ and $ c_4 \approx .068412392872$. There might be some way using a more traditional approach that would get us the true value of these constants.

The $d=2$ case was first worked out by a computer using a different approach. For more information on this, see \cite{EkhadKing}. Their approach expresses the quantity as a double contour integral and applies their own automated techniques to evaluate it. For information on the techniques, see \cite{ApagoduZ}. A completely human produced analysis of this sequence proves more illusive.

\section{Application to Other Sequences}

Our technique also allows for computing binomial transforms of interesting sequences. An example of this is if we were to let $F_k$ denote the $k$-th Fibonacci number and consider the sequence

\[
x_n = \sum_{k=0}^n F_k \binom{n}{k},
\]
we immediately receive that the recurrence that defines $x_n$ is $0=(-N^2 + 3N - 1)x_n$. This recurrence is identical to the recurrence given for (A001906) which is the sequence describing the sum. Though this is already a known fact, if we just bump the power up on $F_k$ to $F_k^3$, we still get a rather nice recurrence relation for the sum, in particular it is described by $0=(-N^4 + 7N^3- 9N^2-2N+4)x_n$. This integer sequence is recently added as number(A298591 \cite{OEIS}). All powers of Fibonacci follow this nice pattern that a linear recurrence where the terms do not depend on $n$ suffices, instead of in general for our technique, where the recurrence may need rational functions of $n$ showing up to describe the next term. We would know for free that it must be P-recursive, but it is an interesting conjecture that it also need be C-finite. These C-finite sequences are discussed in greater detail in \cite{ZCFinite}. The techniques given in that paper can also be applied to some of the problems considered here.

Also of interest, suppose that $a_k$ is be the m-Fibonacci sequence, defined as $a_0=0$, $a_1=1$, and $a_{k+2} = ma_{k+1} + a_k$ for $k\ge 0$. Then, since the program was implemented in a symbolic way, doing this type of problem is no more work than the ones we've already considered. 
\[
x_n = \sum_{k=0}^n a_k \binom{n}{k},
\]
we have 
\[
x_{n+2} = (2+m)x_{n+1} - x_{n}.
\]

This is also known, but is the main theorem of a twelve page paper by Falcon and Plaza \cite{FalconPlaza}.

Since so many sequences of interest in combinatorics already have recurrences found for them, this makes our technique even more powerful. All we need is a recurrence for the non-hypergeometric factors living inside the summand, and then we are set for getting an automatic answer. For example, suppose that you took the Motzkin numbers $M_n$ (A001006) which appears in many guises, but one way of defining it is as the number of (classical) Dyck words on $U$ and $R$ that avoid $UUU$. Among the many forumlae that are known for it is

\[
M_n = \sum_{k} \frac{n!}{k! (k+1)! (2n-k)!}  .
\]

We can then apply the original Celine's algorithm to obtain that $M_n$ satisfies the recurrence:

\[
(n+2) M_{n+2} = (2n+1)M_{n+1} +(3n-3)M_n .
\]

Then, once there is a recurrence to work with, there are so many possible summation that open up to us. For example, if we wanted to describe the sum

\[
x_n = \sum_k M_k^2 \binom{n}{k} ..
\]

After the computer takes several minutes to think, it is able to show that this sequence satisfies the recurrence

\begin{align*}
0=&(-120n^4-1240n^3-4440n^2-6440n-3120)x_n\\
    &+(132n^4+1520n^3+6320n^2+11300n+7368)x_{n+1}\\
    &+(18n^4+318n^3+2044n^2+5660n+5712)x_{n+2}\\
    &+(-33n^4-578n^3-3761n^2-10760n-11400)x_{n+3}\\
    &+(3n^4+61n^3+458n^2+1500n+1800)x_{n+4}.
\end{align*}

Judging just by how complex the expression is, it seems unlikely to be easily discovered by classical (non-automated) approaches. And even though its messiness is a barrier to gaining much human understanding of $x_n$, as already mentioned, there are automated tools of Birkhoff and Tirijinski \cite{WZ} to take recurrences and then distill out facts that are of human interest (such as asymptotics).

\section{Application to Multiple Summations}

Another promising application of our technique is to evaluating multiple sums over hypergeometric terms. A toy example of this would be computing

\[
\sum_{i=0}^n \sum_{k=0}^i \binom{i}{k} \binom{n}{i}.
\]

To find a recurrence for this sequence, pick out any of the factors which contain $k$, and run some automated process to evaluate single summation such as the Zeilberger Algorithm \cite{AeqB}. Often, this sum will not have a nice formula, so we are left with a possibly high order recurrence describing it. However, that is precisely what the techniques here are made to handle, so we can feed this partial evaluation into the procedure. Given enough computing this allows any number of summation signs to be dealt with. For each summation, we have the usual requirements of the original Sister Celine's method, namely that for each summation, the boundaries extend as far as the terms can be without becoming zero. In this particular case, evaluating the inner sum yields $0=(N - 2)x_n $, and substituting in that recurrence, we get that the whole sum satisfies $0=(N -3)z_n$. Which is to say, the sum evaluates to $3^n$. Though this has a nice combinatorial proof counting the number of assignments from $\{1,\ldots,n\}$ to $\{1,2,3\}$ by first picking the $k$ elements that map to either $1$ or $2$, and then, from those $k$ elements, picking the $i$ elements that map to $2$. That requires a moment of thought where such a simple recurrence for the computer only requires less than a second of `thought'. Alternatively, consider the harder problem, where we would want to compute

\[
\sum_{i=0}^n \sum_{k=0}^n \binom{i-k}{k}^2 \binom{n}{i}.
\]

It \textbf{may} be possible to evaluate this in a more human way, but for the computer it can easily determine that the solution is described by the recurrence

\[
0 = (-(n+9)N^5+(7n+54)N^4-(17n+103)N^3+(21n+97)N^2-(15n+50)N+5n+5) x_n.
\]
A Maple package for multiple summations has already been described in \cite{ApagoduZ} and is available at:

{\tt http://sites.math.rutgers.edu/\~{}zeilberg/mamarim/mamarimhtml/multiZ.html}

However our package takes roughly the same time on the simple first example given, and is faster than their package on the second example. Their package, however, gives a `better' analysis of the summation, in that it does indefinite summation, and does not require that on the bounds of summation, the summand is zero. That is, theirs generalizes Zeilberger's algorithm, instead of Sister Celine's method.

\section{Example Usage of this Maple Package}

Hopefully by this point, the usefulness of our package has been made clear. Though there is more detailed documentation in the maple package itself, here is a brief description of how they are used. The first step is to figure out the recurrence that is satisfied by $a_k$, called rec1. Then, call {\texttt findrec(I,J,timeout,rec1,F,d,n,N)} where both rec1 and the output are in shift operator notation, with N denoting the shift operator. This call will attempt to find the recurrence for the sum:

\[
x_n = \sum_{k=0}^n a_k^d H(n,k),
\]
where the recurrence is of order at most $I$, and degree at most $J$. {\texttt Timeout} is the most time (in seconds) to wait on a particular attempt, if it exceeds that time, the procedure exits.

There is a complete description of how to format these problems to make use of the maple package are included in its documentation. The package can be found at {\tt http://sites.math.rutgers.edu/~{}ajl213/DrZ/Celine/RecSum.txt}. Just read the package in and type ``Help()'' to see the documentation.

\section{Future Directions}

The techniques here could easily be extended to allowing arbitrary products of factors, each of which satisfy a known recurrence. This would require only a little bit of modification of these techniques, but simply has not yet been implemented.

\chapter{A Practical Variation on Gosper's Algorithm}
\label{Gosper}

\section{Background}

Gosper's algorithm \cite{Gosper} gives a technique for ``solving" summations of hypergeometric sequences, that is, given a hypergeometric $F(n,k)$, and a sum of the form

\[
x_N = \sum_{n=0}^{N-1} F(n),`
\]

it is able to find a formula for $x_N$ as a constant plus a hypergeometric term, if one exists. For a detailed description of how (and why) Gosper's algorithm works, we refer the reader to chapter five of the book \underline{A=B} \cite{AeqB}.

However, many sequences of interest are not hypergeometric, so there is definitely room for further work along this same goal of computing indefinite summations.

\section{Overview}

Instead of starting with sequences more complicated than hypergeometric, we will start with applying our approach to sequences that are less complicated, and move up from there. Consider rational functions. Suppose that we have some summation 
\[
x_N = \sum_{n=0}^{N-1} \frac{P(n)}{Q(n)}
\]

and we want to say what rational function this is equal to. We know that one exists because we could always of expanded out the summand by rational functions that telescope, notice some cancellation, and then add together what is left over to get a rational expression. In Gosper's algorithm, we are able to tell precisely when some summation is equal to a constant plus a hypergeometric term. However, to pay the price for expanding the possible values that it could sum to, we will have to give up this guarantee of knowing for certain that if our algorithm fails to figure out such a summation, then there is none. Instead, it could be that it only failed for the considered degree of the recurrence, and it may instead find a recurrence that the summation satisfies by simply increasing the order of the recurrence or the degree of the rational functions that are  used as coefficients in the recurrence.

First, we try to identify an expression for the limit $L$ of the summation. For our considered summations of a rational summand, we know that the limit will always be rational, and so, by looking at partial fraction decompositions of larger and larger partial summations. Looking at the decomposition, one entry blows up while everything before that entry stays the same, giving us a very good guess that the limit is the part of the decomposition before the entry blowing up. Then, once we know the value of the limit of the summation, we construct the sequence

\[
y_N = \sum_{N}^{\infty} \frac{P(n)}{Q(n)},
\]
that is,

\[
y_N = L - x_N
\]

Then, we generate may terms of this sequence and try to identify the result as a hypergeometric sequence. This step in particular allows for a lot of freedom. If instead of hypergeometric, we were looking for descriptions of this sequence as a higher order recurrence, we could just use that information to change the set of linear equations that need to be solved. The other class of functions that we are considering are called multibasic sequences. They are ones where the ratio of successive terms is some fraction of multinomials in different expressions depending on $n$ instead of just polynomials in $n$. 

So, for example, to evaluate

\[
\sum_{n=0}^{N}-\frac{1}{4}\frac{(2^{2n+2}n^2+2n2^{3n}-2^{4n}+3n2^{2n}+n2^{n+1}+4n^2-2^{3n}-72^{2n}-42^n-12)}{(n!(2^{2n-2}+1)(2^{2n}+1))}
\]

It is found to be

\[
e + \frac{(2^N+N+3+2^{2N})}{((2^{2N}+1)N!)}
\]

That is, it is Euler's constant plus a multibasic expression in $N$ and $2^N$ of max degree 2.

The drawback of this technique is that of identifying $L$. It works fine if the function is very quickly converging, but it starts to perform poorly as the sum converges more slowly, and is completely worthless when the sum does not have a finite limit. Instead of solving the set of linear equations that we set up, also allow this limit to be an unknown, and solve a resulting non-linear set of equations. Of course this is much slower, and so this techniques is only best used for summations that rapidly converge.

Another broad class of summands to which this can apply are not just when we extend the rato of successive terms to some fixed rational expression of atoms that are hypergeometric, 

\[
\frac{P(a_1,a_2,\ldots,a_k)}{P(a_1,a_2,\ldots,a_k)}
\]

where $a_i$ is something such as $2^n$ or $n!$. Instead, what if we were to allow $a_i$ to be more free in how it depends on $a_i$. For example imagine that we had that successive terms had the ratio $\frac{F_{n+1}}{F_{n}}$ where $F_n$ for this little example represents the $n$th Fibonacci number. More generally, the structure about this that is helpful to us is that our new atoms that we are exploiting is that they are $C-finite$. This allows us to reduce occurrences of them for larger $n$ to a number of starting terms equal to the order minus one, possibly times rational expressions in $n$ that come from the recurrence that they satisfy. 

\section{Examples and Motivations}

Many of the examples here are constructed in order to show the usefulness of this procedure, instead of arising organically. Any, In general, non-cooked up examples would also yield results. The only problem with that is that the degrees that would be required for for the solution could be very very high. There is not a way to know before hand how high of degrees would need to be considered, so our procedure needs to look for higher and higher degree solutions.

The example that initiated this whole investigation was the monthly problem that asks the reader to compute the value of
\[
\sum_{n=0}^{\infty} \frac{1}{(n^4+n^2+1)n!}.
\]

The immediate impulse of expanding using partial factions does not work, since we cannot factor the bottom in such a way to get the terms to telescope and leave us with our answer. The other go-to resource for evaluating sums, Maple, is also helpless when faced with this problem. All it does is parrot back the sum. Mathematica managed to compute the value of the infinite sum, but gave an expression far more complicated than necessary to express the value of the partial sums. However, the trick for finding a simple expression is to notice that if we took our sum and subtracted off 

\[
\sum_{n=0}^{\infty} \frac{1}{2} \frac{1}{n!},
\]
Then, the sum becomes
\begin{align*}
\sum_{n=0}^{\infty}\left(\frac{1}{n^4+n^2+1} - \frac{1}{2}\right)\frac{1}{n!}&=\sum_{n=0}^{\infty}\left(\frac{2}{2n^4+2n^2+2} - \frac{n^4+n^2+1}{2n^4+2n^2+2}\right)\frac{1}{n!}\\
&=\sum_{n=0}^{\infty}\frac{1-n^2-n^4}{2n^4+2n^2+2}\frac{1}{n!}\\
&=\sum_{n=0}^{\infty}\frac{1}{2}\left(\frac{-n}{n^2-n+1} + \frac{1+n}{n^2+n+1}\right)\frac{1}{n!}.\\
\end{align*}

Since these two terms telescope, that is, plugging in $n:=n+1$ into the first gets you the second, we have that the value of the partial sum is given just by

\[
\frac{1+N}{N^2+N+1}\frac{1}{N!}.
\]

So, once we have subtracted off $\frac{1}{2}e$ in a fancy way, the resulting sum clearly goes to zero.

When we were drawing inspiration from this problem, as mentioned earlier, we gave up hope that the algebra would always work out so nicely as this toy problem. The part that we kept from this motivation was that it required some sort of guessing at the value of the limit. Once a value for the limit is guessed, we find a summation that goes to that value and matches factors found in the input summation. So, even though there is initially guessing at the value of the limit, once the guess is made, it is possible to go back and prove that that guess was correct. 

For the case of this problem, the system is able to correctly simplify the hypergeometric representation that it finds of the partial sum to this equally simple expression. This is far better than Mathematica's result for the partial sum which is littered with complex hypergeometric expressions and takes about half a page to display. 

Some more complicated sums that you might encounter that this same approach can consider are ones such as:

\[
x_n = \sum_n \frac{n^2+1}{(2n)!}
\]

Then, without any changes, the same experimental approach that solved for us the last problem gets us that the sum appears to be converging to:

\[
5/4 \cosh(1) + 1/4 sinh(1)
\]

Then, subtracting off the taylor series for that, it is able to identify a hypergeoemtric ratio that seems to be satisfied as

\[
\frac{x_{n+1}}{x_n} = \frac{n+1}{ 	2n^2(n+1)}
\]

Once it has figured that out, it is then able to find a closed form expression for this indefinite summation

\[
x_n = \frac{1}{2} \frac{n+1}{(2n+1)!}
\]

once it has found this (with guessing along the way) it is able to easily verify that this expression does describe the value of the sum, since to check, all you need to do is subract $x_n$ from $x_{n-1}$ and verify that you get the summand

As a stress test, we may even consider some sum as messy as:

\[
x_n = \sum_k\frac{36n^9-126n^8+489n^7-1343n^6+1633n^5-784n^4-582n^3-310n^2-735n+237}{(n^3-2n^2+11n-3)(n^3+n^2+10n+7)(2n)!}
\]

Then, the procedure is still able to make progress towards an answer. It makes a guess for the value of the sum of

\[
11 cosh(1)
\]

and is then able to guess a value for the hypergeometric ratio.

\[
\frac{x_{n+1}}{x_n} = \frac{ (9 n^4  + 3 n^2  + 2) (n^3  - 2 n^2  + 11 n - 3)}{2 n (2 n - 1) (n^3  + n^2  + 10 n + 7) (9 n^4  - 36 n^3  + 57 n^2  - 42 n + 14)}
\]

and finally outputs that a formula for $x_n$ is 

\[
x_n = \frac{ 9 n^4  + 3 n^2  + 2}{ (n^3  + n^2  + 10 n + 7) (2 n)!}
\]

which can then be rigorously, automatically checked.

Of course these examples are somewhat cherry picked as you would not expect a typical random summation that you would want to compute to even have a Hypergeometric representation. We saw this fact in the previous chapter where we were instead hunting for recurrences to describe the summation. Even though many sums would not admit a hypergeometric solution, for those that do, it is a much better description of the indefinite summation. 

Another major limitation of this technique is that it requires the summand to be very quickly decreasing to zero. This is because the guessing procedure that is used in order to figure out how much to shift the summation by relies on being able to compute out many, many digits of precision for the sum. It gets an estimate of how many digits it can trust by computing the partial sums out to different lengths, and then seeing up though how many digits the two approximations agree.

\section{Using this Maple Package}

Hopefully by this point, the potential usefulness of these procedures has been made clear. Though there is more detailed documentation in the maple package itself, here is a breif description of how they are used. To try and figure out a hypergeometric expression for the indefinite sum, allowing degree at most $d$, and starting at $n=0$, call $PMG(expression, n,d,0)$. The procedure will then spit out its guess for the summation, followed by a line that contains a constant followed by the ratio of consecutive terms satisfied by the sequence. Lastly, it will attempt to return a closed form expression for the summation, which may often not exist.

For The version that is multibasic, instead call $PMGMB$, with the same format for arguments. The main difference however is that we now allow for more complicated expressions as outputs than just hypergeometric expressions. This has been more thoroughly discussed in an earlier section.

For both of the packages, the guessing system occasionally requires some hints. This can be given in the form of an optional last argument, where you list atoms that you might expect to appear in the answer. For example, if the summation converged to $e +e^{-1}$ you could just as well say that it converges to $2sinh(1)$. As a human, you would need to inspect the summation to see if there is a $(2n)!$ factor, which would indicate that the latter is a better way to interpret the number that it is converging to, and pass in the hint $[sinh(1)]$ as an argument.

\section{Future Directions}

While this procedure helps to solve more types of sums, there is very limited application of it to computing many more. Apart from its usefulness to solve sums that couldn't be gotten using the existing techniques, it also reexamines sums that could already be solved from a much, much more simple approach than Gosper's algorithm does. This simplicity is a benefit in and of itself, even though it does not have the same assurances of completeness that Gosper's algorithm does, it is no less rigorous in the correctness of its answer if it is able to find a solution. In principle, there is some selection of degrees and basic sequences that would cause this procedure to find a formula for a convergent sum. However, finding some finite bound on this search space, if one exists at all, is a good candidate for future investigation.

\chapter{Bunk Bed Conjecture}
\label{BBed}

\section{Introduction}

The bunk bed conjecture is one of the many problems in percolation which seems obvious, and yet is elusive. It was mentioned as early as 1985 by \cite{FirstBBC} in a slightly different form and more formally by \cite{SecondBBC} in which it is already acknowledged as folklore.

There are many notions of closeness in graphs, and it being true lends credence to a notion of closeness of random graphs which seems natural. In particular, we describe two vertices in a graph as being closer if they have higher probabilities of being connected. By making an identical copy of the graph that is connected and adjacent, by an intuitive sense, we would think of a vertex that is in the other copy as further away. In fact, in the conventional, non-random definition of distance in a graph it is precisely distance one more away. Since the copy of the graph is identical, it holds the most hope of not having some strange behavior introduced by the candidate notion of closeness in a random graph. Proving this conjecture confirms that considering likelihood of being in the same component as a notion of closeness agrees with this intuition. A different notion which also seems reasonable, where we look at the expected hitting time under a random walk does not agree with this \cite{SecondBBC}, even though at face value it seems like it would of been a reasonable notion of closeness.

As a very simple example, consider that we pick out graph $G$ to be a copy of $K_2$, then, when we compute $G\square K_2$, we get $C_4$, then out start vertex $s$ and our ending vertex $f$ are in the same component with probability $p + p^3 - p^4$ whereas $s$ and $f'$ lie in the same component with probability $p^2 + p^2 - p^4$, which is less than the former quantity regardless of the probability $p$ that any particular edge is retained because $1+p^2 \ge 2p$.

There has only been piecemeal progress in proving it. In \cite{BBedComplete} they show that the conclusion of the conjecture is true for complete graphs, and even then, only when we fix the probability of retaining an edge $p=\frac{1}{2}$. In \cite{BBedEquiv} they give a few simple relations between the original conjecture and related conjectures. In \cite{Generalizations} they build on these, relating the original conjecture to even more general seeming conjectures, via looking at some complicated generalizations of the conjecture to randomly directed graphs, they are able to recover that the original conclusion is true for outer planar graphs. Our approach does not depend on the structure of the original finite graph $G$.

\section{Definitions}

The Cartesian graph product $G_1 \square G_2$ is defined as the graph on the vertex set $V(G_1)\times V(G_2)$, where there is an edge between $(u,v)$ and $(y,w)$ if either
\begin{itemize}
\item{ $u=y$ and there is an edge between $v$ and $w$ in $G_2$  \bf{or}}
\item{ $v=w$ and there is an edge between $u$ and $y$ in $G_1$}
\end{itemize}

For this problem we will concern ourselves just where the case that $G_2=K_2$ with vertex set $\{0,1\}$. We will write $\{(v,0)\}_{v\in G_1}$ as $G$ andl $\{(v,1)\}_{v\in G_1}$ as $G'$. Similarly, for $v\in G$, we'll let $v'\in G'$ be the vertex that is obtained by flipping the 0 to a 1. Also, for convenience, let $v'' = v$. Our graph product places an edge between them. We'll call an edge going between $G$ and $G'$ an outside edge. The goal is going to be to prove that the probability $s\in V(G)$ is connected to $f\in V(G)$ is greater than or equal to the probability $s\in V(G)$ is connected to $f' \in V(G')$. As mentioned in \cite{Generalizations} we may assume that the outside edge from $s$ to $s'$ and from $f$ to $f'$ are absent. Assume this throughout. 

We will call a particular choice of which edges remain, and which do not, a configuration.

For each configuration of edges there will be some set of outside edges that remain.  We will show that conditioning on this set of outside edges, then the conclusion of the bunk bed conjecture still holds. That is, that the probability that $s$ is connected to $f$ is at least the probability that $s$ is connected to $f'$. Let $X\subset V(G)$ be the set of vertices that are incident to an outside edge that has not been removed. If we show that the conjecture holds for each possible $X$, then when we combine them together, weighting by the probability that that $X$ occurs, then we will of shown it for the original problem, not conditioned on a particular $X$. However, our approach only works for the situation that $|X|=2$, an improvement on that it is known for $|X|=1$.

We introduce the events $A_{u,v}$ for any $u,v\in V(G)$. We say that $A_{u,v}$ occurs if there is a path from $u$ to $v$ so that all of the path, except possibly the endpoints lie in $V(G)\setminus X$. Similarly define $A'_{u',v'}$ for paths going between vertices of $G'$. Notice that $\{A_{u,v}\}$ are positively correlated and $\{A'_{u',v'}\}$ are positively correlated, both by Harris's inequality \cite{Harris}. Also notice that since each $A_{u,v}$ and $A'_{u',v'}$ are either true or false based off of which edges remain from $E(G)$ and $E(G')$ respectively, they are independent. In fact, in order to maintain as much generality as possible, we will only be using they are not positively correlated, that is, for $A\subseteq \{A_{u,v}\}$ and $B\subseteq \{A'_{u',v'}\}$, we have $P(A | B) \le P(A)$. The last property that we will use frequently is a symmetry property, that is for any $S_1,S_2\subseteq\{A_{u,v}\}$, if $S_1'$ means just replacing each $A$ event in $S_1$ with the corresponding $A'$ event, then we have $P(S_1 | S_2) = P(S_1' | S_2')$.

Now, we will further condition the problem and show that the conclusion of the bunk bed conjecture still holds. Define the shadow of a path $p$ on $G\square K_2$ to be the walk obtained by replacing each $v'\in V(G')$ on the path with $v$, and removing any duplicates that arise because the path was traveling along an outside edge. Note that the shadow of a path is not necessarily a path, as the original path may use both a vertex in $x\in G$ and $x'$. Assume that there is some path $p$ from $s$ to either $f$ or $f'$ in $G\square K_2$ whose shadow is $p_s$. We are assuming that such a path exists because if no such path exists, then the configuration in question trivially has the conjecture hold in that both probabilities are zero. Given this path, look at its shadow in $G$,  call that path $s = y_0, y_1, \ldots, y_k=f$. Now, let $m = |\{y_i\}_i \cap X|$. Define $x_0=s$, $\{y_i\}_i \cap X = \{x_1,x_2 ,\ldots, x_m\}$, and finally $x_{m+1} = f$.


\section{Main Proof}

We will prove that the statement of the bunk bed conjecture is true if $|X|=2$

If $s$ is connected to no vertices in $X$, then there is no way for $s$ to be connected to $f'$, so the statement is trivially true.

Fix any subset $\emptyset \neq Y\subseteq X$ of vertices that are connected to $s$. To say that $s$ is connected to $Y$ is equivalent to requiring that there is some spanning tree $T = \{\{v_i,w_i\}\}_{i=1}^{|Y|}$ on the vertices $\{s\}\cup Y$, so that for each $i$, either $A_{v_i,w_i}$ or $A'_{v_i',w_i'}$ occurs. Also, any time that one of the two vertices is $s$, we need that it needs to be the $A$ event, not the $A'$ event that happens, because there is no outside edge on $s$. Then, we have that there is some $Y_2\subseteq Y$ so that $\{A_{s,y}\}_{y\in Y_2}$ occurs. We also see that for every $z \in Y$,

\[
P(A_{z,f} | \{A_{s,y}\}_{y\in Y_2}) \ge P(A_{z,f}) = P(A'_{z',f'}) \ge P(A'_{z',f'}|\{A_{s,y}\}_{y\in Y_2})
\]

The first inequality because of positive correlation, the second by symmetry. The last inequality is actually an equality for the statement of the bunk bed conjecture that we have mentioned so far because the edges remaining in one copy of the original graph do not influence the edges remaining in the other. We leave it as an inequality here because we will only have inequality in an extension mentioned later. 

All of the other events that we are conditioning on for vertices of $Y$ to be connected to $s$ are symmetric, because they are of the form $A_{v_i,w_i}$ or $A'_{v_i',w_i'}$ occurs. And since the asymmetry is already on the side of requiring more $A$ events than $A'$ events, it is at least as likely for each of these conditions that we have $A_{v_i,w_i}$ as $A'_{v_i',w_i'}$, which only further tips the scales for the $A$ events. Put more concretely, suppose that we have the events indexed by $I_1\subseteq [|Y|]$ occurring as $A$ events and not as $A'$ events, and $I_2 \subseteq [|Y|]$ occurring as $A'$ events but not as $A$ events. We know that $|I_1|\ge 1$ by above comments. Let $I_3$ be the times that both occur. Then, we have by positive correlation that:

For ease of notation, let $B_I = \{A_{v_i,w_i}\}_{i\in I}$ and $B'_I = \{ A'_{v'_i,w'_i}\}_{i\in I}$. Let $S_1$ be the indices of edges of $T$ incident to $s$, and $S_2 = [|Y|]\setminus S_1$. Then, put into these symbols, the event that $s$ reaches all of $Y$ is identical to there being a $S_1\neq \emptyset$ so that $B_{S_1} C_{S_2}$ happens.


\begin{lemma} \label{BBedLemma}
Suppose that we have two sets of $A$ events $D_1$ and $D_2$, then, for any $B_i$, $P(D_1 |(B_i\vee B_i') D_2) \ge P(D'_1 |(B_i\vee B_i') D_2)$
\end{lemma}
\begin{proof}

\begin{align*}
&P(D_1  (B_i\vee B_i') D_2) -  P(D'_1  (B_i\vee B_i') D_2)\\
&=P(D_1 B_i D_2 ) - P(D'_1 B_i D_2)   +  P(D_1 B'_i D_2) - P(D'_1 B'_i D_2) \\
&\,\,\,\,\,- P(D_1 B_iB_i' D_2) +   P(D'_1 B_iB_i' D_2)\\
&=P(D_1 B_i D_2 ) - P( B_i D_2)P(D'_1)   +  P(D_1 D_2)P(B'_i) - P( D_2) P(D'_1 B'_i) \\
&\,\,\,\,\,- P(D_1 B_i D_2)P(B_i') +   P( B_i D_2)P(D'_1B_i')\\
&=P(D_1 |B_i D_2  )P(D_2|B_i) P(B_i) - P(  D_2| B_i) P(B_i ) P(D_1)   \\
&\,\,\,\,\,+  P(D_1 D_2)P(B_i) - P( D_2) P(D_1  | B_i) P(B_i) \\
&\,\,\,\,\,-P(D_1 |B_i D_2  )P(D_2|B_i) P(B_i)P(B_i) +   P( D_2|B_i) P(B_i) P(D_1B_i|C_I)P(B_i)\\
\end{align*}

Everything has a factor of $P(B_i)$ so divide through by that, and combine the first and fifth terms.
\begin{align*}
&=P(D_1 |B_i D_2  )P(D_2|B_i)(1- P(B_i)) - P(  D_2| B_i) P(D_1)   \\
&\,\,\,\,\,+  P(D_1 D_2) - P( D_2) P(D_1  | B_i)  +   P( D_2|B_i) P(B_i) P(D_1B_i)\\
&\ge P(D_1 |B_i)P(D_2|B_i)(1- P(B_i)) - P(  D_2| B_i) P(D_1)   \\
&\,\,\,\,\,+  P(D_1 D_2) - P( D_2) P(D_1  | B_i)  +   P( D_2|B_i) P(B_i) P(D_1B_i)\\
&= P(D_1 |B_i )P(D_2|B_i)- P(  D_2| B_i) P(D_1)  +  P(D_1 D_2) - P( D_2) P(D_1  | B_i)  \\
&\ge P(D_1 |B_i )P(D_2|B_i)- P(  D_2| B_i) P(D_1)  +  P(D_1)P(D_2) - P( D_2) P(D_1  | B_i)  \\
\end{align*}

Dividing through by $P(D_1)P(D_2)$, and factoring, this expression becomes

\begin{align*}
&= \left(\frac{P( D_i | B_i)}{P( D_1 )} - 1\right)\left( \frac{P(D_2  |B_i)}{P(D_2 )} - 1\right)\\
&\ge 0
\end{align*}

Throughout these steps, we are making heavy use of the fact that $A$ events and $A'$ events are independent of one another, that $A$ events are positively correlated, and that $A$ events and $A'$ events are symmetric.

\end{proof}

We have that $s$ is connected to $f$ if at least one of $A_{s,f}, \{A_{z,f}\}_{z\in Y}$ occur, and it is connected to $f'$ if at least one of $\{A'_{z',f'}\}_{z'\in Y'}$ occur. This is because we know that the $s$ is not connected to any other vertices in $X$ than $Y$, and so any of the paths from $s$ to $f$ or $f'$ must last hit $X$ at a vertex in $Y$. Having the part of that hypothetical path that is past the last time that it is in $X$ is exactly captured by either an event in $\{A_{z,f}\}_{z\in Y}$ or an event in $A\{A'_{z',f'}\}_{z'\in Y'}$ being true. If $|Y|=1$, then we an just apply Harris's inequality. For the case of $|Y|=2$, then either $s$ is connected to both using edges in $G$ (just $A$ events), or there is just one vertex in $Y$ which is hit by an A event from $s$. The other one is connected to that vertex by either an $A$ event or an $A'$ event. This places us exactly in the conditions of Lemma \ref{BBedLemma} where $D_2$ is the $A$ event getting us to $Y$, $(B_i \wedge B_i')$ is what is connecting the two vertices of $Y$, and $D_1$ is any one of $A_{s,f}, \{A_{z,f}\}_{z\in Y}$.Since we are assuming that $s$ is connected to all of $Y$, we use our inequality over all $z\in Y$ to get the conclusion that it is at least as likely that $s$ is connected to $f$ as that $s$ is connected to $f'$. Since this is true for any choice $X$ and $Y\subseteq X$, then it is true when we instead do not condition on these things by just summing over all the choices for X and Y. 





\section{Extensions and Related Results}

While still conditioning on $X$, we cannot have a hope of proving that the inequality is strict. For example consider $G=P_3$, the path with three vertices, and let $X$ consist of the middle vertex. Then, picking $s$ and $f$ to be the ends of the path, we have equal likelihood that $s$ connected to $f$ as having $s$ connect to $f'$. 

This proof also gives us information on which values of $X$ will have the two probabilities be equal. In particular, we want that for each step of the induction, when we multiply by $P(L_{i+1}|L_i) - P(L_{i+1})$, we are not multiplying by zero. Everything thing else that we do would preserve a strict inequality that we start with since $(R_0)<1$. So, we need that for some path shadow $p_s$, for each $i$, the set of edges which $A_{x_{i-1},x_{i}}$ and $A_{x_i,x_{i+1}}$ are not disjoint. as we have a strictly positive correlation between two events that have some coin that they depend on in common.

\begin{conjecture}
A graph with a known collection of vertices with outside edges $X$ will satisfy the bunk bed conjecture strongly if and only if there is a path from $s$ to $f$ in the subgraph of $G$ induced by removing the vertices in $X$.
\end{conjecture}


Since we showed at least a weak inequality for any selection of $|X|=2$ for the cross edges, we can get strong inequality in the original statement of the conjecture if you assume at most . In the original bunk bed conjecture, we have a non-zero probability that $X = \emptyset$. For this $X$, then the probability that $s$ is connected to $f'$ is zero, while the probability that $s$ is connected to $f$ is positive because there is some chance that we retain a spanning tree of $G$, for example.

Our approach seems like it would also generalize to being able to prove the stronger conjecture (2.5) presented in \cite{Generalizations} (Still under the assumption of at most two outside edges). In this paper they introduce a model of randomness we loose independence between the random choices made on the edges of $G$ and on the edges of $G'$. Here, they flip coins randomly for one of the two copies of the graph, say $G$. Then, to determine the edges of $G'$, they include an edge if and only if it is not remain in $G$. Note that this will mean that $\{A_{u,v}\}$ and $\{A'_{u',v'}\}$ are negatively correlated. The statement of our lemma is actually able to be a bit simpler, since in this context, we have that $B_i$ and $B_i'$ are mutually exclusive, so that we do not have to track the term that we got in that proof from inclusion-exclusion.

In \cite{Generalizations} they also present a form of the conjecture formulated in terms of hypergraphs. It is doubtful that the approach that we have given here would be useful to proving that different conjecture, since there is not as simple of a way of breaking down the notion of ``connected to'' that they use in terms of the existence of paths of retained hyperedges. It seems like it might be fruitful to further investigate if any of these ideas here can be related to this similar statement of the conjecture.

By looking at polynomials produced by going back to the model where we randomly remove with probability $p$ and taking the difference of the likelihood that we are connected to $f$ and to $f'$, we are able to associate a polynomial to the graph in an interesting way. In particular, by a recursive approach, we are able to get that the formula depth for the polynomial is at most $O(E(G))$. That is, for each edge that we process, we either place p times the polynomial that we get from definitely including the edge plus $(1-p)$ times the polynomial from definitely excluding the edge. Finally as a base case, we associate a 1 to any configuration which has $s$ connected to $f$ but not $f'$. Associate a -1 to any configuration which has $s$ connected to $f'$ but not $f$. To any other configuration, associate 0. This allows us a somewhat quicker way that we used of verifying the conjecture for small graphs. Actually expanding out this polynomial to make sure that it is nonnegative for any p in $[0,1]$ still is a quite slow operation.

\section{Experimental Component}

As is always good when faced with a conjecture, it is worthwhile to check small cases, even if it seems obvious from an intuitive point of view. Using the techniques in this section, we checked all graphs up though size six, and may graphs of size 7. While it is always the hope to find a counterexample, as that is much easier to prove and resolve the question, there was no such counterexample. However, The procedure that we used spits out a polynomial that encodes information about the graph. That is, as a polynomial in $p$, what is the difference between the probability that $s$ and $f$ are connected and the probability that $s$ and $f'$ are connected. It is easy to notice the way in which the graph determines some of the polynomial considered. 

The procedure begins by taking two copies of the bunk bed graph. One copy will denote the edges that we have yet to flip a coin for, and the other the edges that are currently remaining. As a base case, if there is a path from $s$ to $f$ or $f'$ in the retained edges minus the edges left to flip, then we have that our probability in question is either $1$, $0$, or $-1$. Once we have obtained the polynomial describing the difference of these two probabilities, it is a very simple task of checking to make sure that it is non-negative on $[0,1]$. Of all the polynomials generated, it  is the case that they have a factor of $p$ to the length of the shortest path in $G$ between $s$ and $f$, and a factor of $1-p$ to the size of the smallest cut set between $s$ and $f$. The rest of the polynomial in all graphs of size at most six is irreducible. This seems kind of surprising, and so is definitely worthy of further investigation.

\section{Maple Package}

An implementation of the procedure used for checking small cases of the conjecture is available at  {\tt http://sites.math.rutgers.edu/\~{}ajl213/DrZ/Bunkbed/bunkbed.txt}. There is documentation there on how to use the package. Just type `Help()' to see an overview of the available functions provided.


\end{document}